\documentclass[12pt, a4paper]{article}
\usepackage{amsmath,amssymb,amsthm}
\usepackage[left=2.0cm,right=2.0cm, top=2cm, bottom=2.5cm]{geometry}
\usepackage{hyperref}
\usepackage{braket}
\usepackage{graphicx}
\usepackage{blindtext}
\usepackage{mathtools}
\usepackage{xcolor}

\usepackage[T2A]{fontenc}
\usepackage[utf8]{inputenc}
\usepackage[russian, english]{babel}
\usepackage{lipsum}

\newcommand{\ee}{\varepsilon}

\theoremstyle{definition}
\newtheorem{theorem}{Theorem}
\newtheorem{lemma}{Lemma}
\newtheorem{proposition}{Proposition}

\newtheorem{conjecture}{Conjecture}
\newtheorem{definition}{Definition}
\newtheorem{example}{Example}
\theoremstyle{remark}
\newtheorem{remark}{Remark}

\title{Feynman checkers: the probability of direction reversal}
\date{}
\author{Ilya Bogdanov,\\ National Research University Higher School of Economics \\ (Faculty of Mathematics)}

\begin{document}
\maketitle

\begin{abstract}
We study the most elementary model of electron motion introduced by R.Feynman in 1965. It is a game, in which a checker moves on a checkerboard by simple rules, and we count the turnings. The model is also known as \textit{one-dimensional quantum walk}. In his publication, R.Feynman introduces a discrete version of path integral and poses the problem of computing the limit of the model when the lattice step and the average velocity tend to zero and time tends to infinity. We get a nontrivial advance in the problem on the mathematical level of rigor in a simple particular case; even this case requires methods not known before.

We also prove a conjecture by I.Gaidai-Turlov, T.Kovalev, and A.Lvov on the limit probability of direction reversal in the model generalizing a recent result by A.Ustinov.

%In addition to the known methods we apply the Legendre polynomials to the model for the first time.
\end{abstract}

\section{Introduction}
We study the most elementary model of electron motion introduced by R.Feynman in 1965. It is a game, in which a checker moves on a checkerboard by simple rules, and we count the turnings.
In his publication \cite{Feynman-1965}, R.Feynman describes quantum theory with path integral formulation --- an approach that generalizes the action principle of classical mechanics. R.Feynman introduces a discrete version of path integral and poses the problem of computing the limit of the model when the lattice step and the average velocity tend to zero and time tends to infinity.
The model was Richard Feynman's sum-over-paths formulation of the Green function for a free particle moving in one spatial dimension. It provides a representation of solutions of the lattice Dirac equation in $(1+1)$--dimensional spacetime as discrete sums.

In the present paper we get a nontrivial advance in the Feynman's problem
\cite[Problem 2.6]{Feynman-1965} on the mathematical level of rigor in a simple particular case: we compute the real part of the discrete Green function approximately when $x = 0$. In this case the Green function depends periodically on time up to remainder that tends to $0$ as time tends to infinity (see Theorem \ref{asymptote}). Even this case requires methods not known before.
 
We also prove a conjecture by I.Gaidai-Turlov, T.Kovalev, and A.Lvov on the limit probability of direction reversal in the model generalizing a recent result by A.Ustinov (see Theorem \ref{new-theorem}).
%In addition to the known methods, we apply the Legendre polynomials to the model for the first time.

\section{Preliminaries}
In this section we recall basic properties of the model; the content of this section is essentially taken from \cite{Preprint}.

\subsection{The basic model}
First, we introduce the path integral in the simplest discrete model. Consider a source point $(0, 0)$ and a destination $(x, t)$. A checker moves to the diagonal-neighboring squares, either upwards-right
or upwards-left. To each path of $s$ of the checker assign a vector $a(s)$ as follows. Take a two-dimensional vector $(0, 1)$. Each time the checker makes a $90^o$-turn, the vector is rotated through $90^o$ clockwise (no matter what direction the checker turns). Finally, $a(s)$ comes as this vector divided by $2^{(t-1)/2}$, where $t$ is the total number of moves (this is just a normalization).
 
Denote by $a(x, t) \coloneqq \sum\limits_{s} a(s)$ the sum over all the checker paths from the source $(0, 0)$ to the destination $(x, t)$ \textit{starting with the upwards-right move}. For instance, $a(1, 3) = (0, -1/2) + (1/2, 0) = (1/2, -1/2)$; 
see Figure \ref{Checker-paths} to the bottom-left. The length square of the vector $a(x, t)$ is called \textit{the probability to find an electron} in the square $(x, t)$, if it was emitted from the square $(0, 0)$. 

To be more rigorous, we give the following definition.
\begin{definition}[\cite{Preprint}, Definition 1] \label{def-basic}
A \emph{checker path} is a finite sequence of integer points in the plane such that the
vector from each point (except the last one) to the next one equals either $(1, 1)$ or $(-1,1)$. A \emph{turn} is a point of the path (not the first and not the last one) such that the vectors from the point to the next and to the previous ones are orthogonal. 

Denote
$$
a(x, t):=2^{(1-t) / 2} i \sum\limits_s(-i)^{\mathrm{turns}(s)}
$$
where the sum over all checker paths $s$ from $(0, 0)$ to $(x, t)$ with first step to $(1, 1)$ and  $\mathrm{turns}(s)$ is the number of turns in $s$. Hereafter, empty sum is $0$ by definition. Denote
$$P(x, t):=|a(x, t)|^{2}.$$
\end{definition}

Figure \ref{Checker-paths} to the right depicts the vectors $a(x,t)$ and the probabilities $P(x,t)$ for small $x,t$.
Figure \ref{fig-arrows} does the same for $t$ up to $50$.
\begin{figure}[htbp]
%{r}{5.2cm}
%\vspace{-0.5cm}
\begin{center}
\begin{tabular}{c}
\includegraphics[width=5cm]{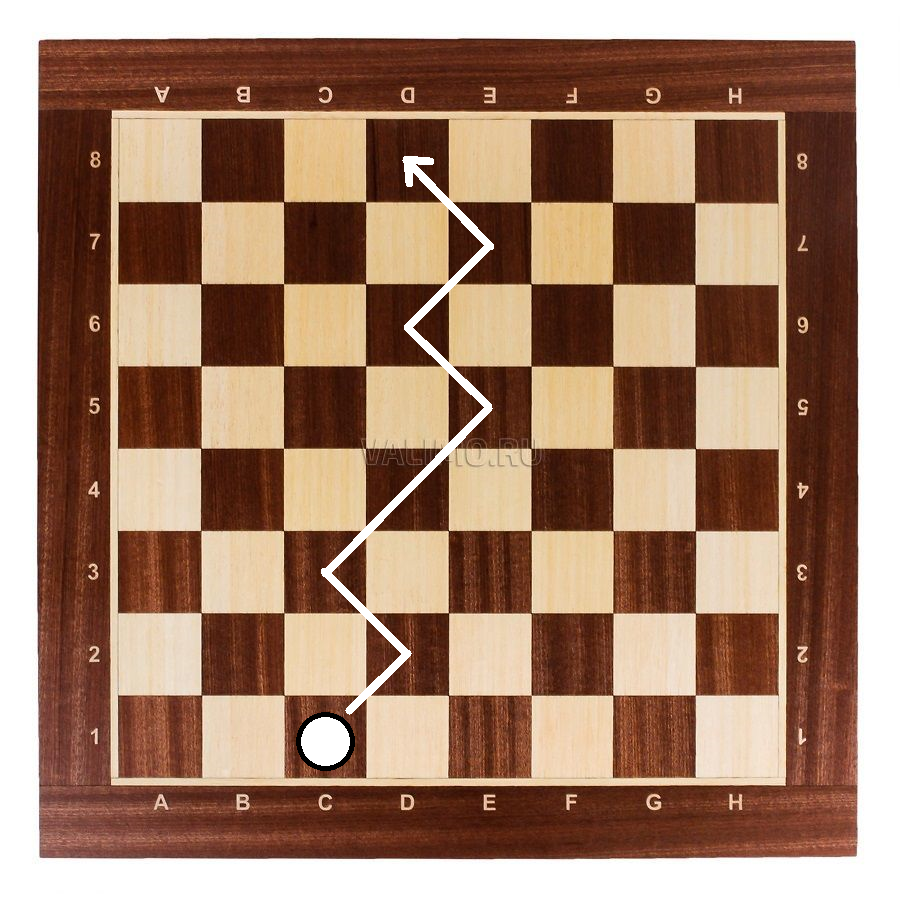}\\
\includegraphics[width=5cm]{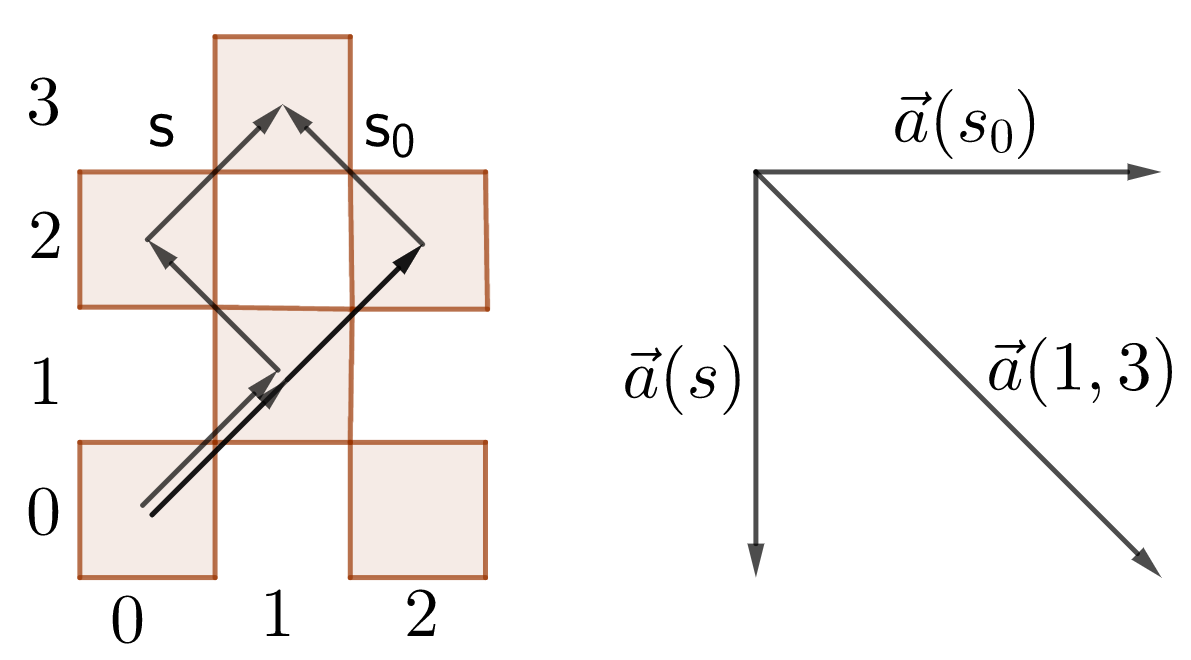}
\end{tabular}
\begin{tabular}{c}
\includegraphics[width=0.58\textwidth]{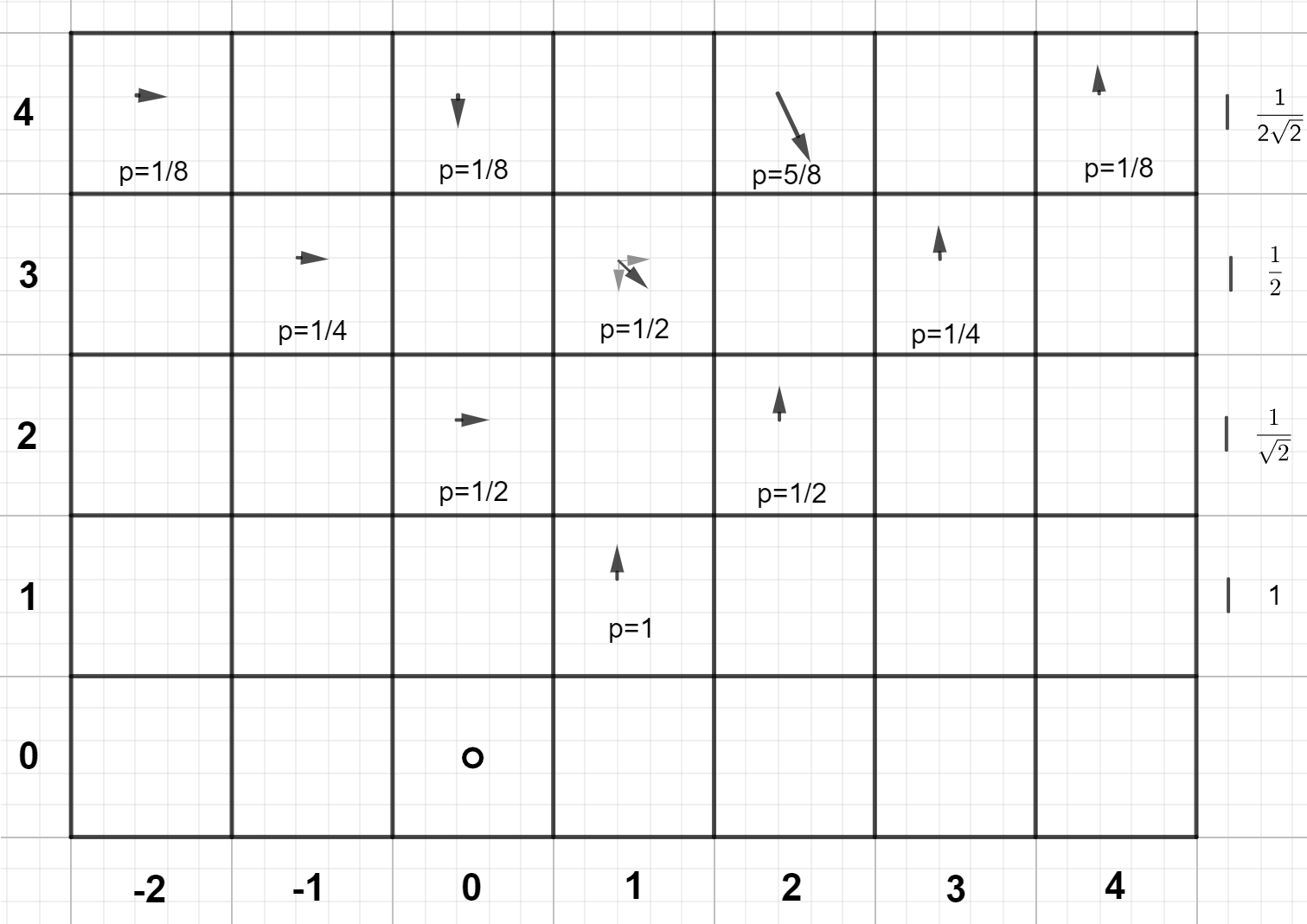}
\end{tabular}
\end{center}
%\vspace{-0.4cm}
\caption{(by V.~Skopenkova) Checker paths (left); $a(x,t)$ and $P(x,t)$ for small $x,t$ (right)}
\label{Checker-paths}
%\vspace{-0.6cm}
\end{figure}

The probabilistic nature of $P(x, t)$ is explained by the following proposition.

\begin{proposition}[\cite{Preprint}, Proposition 2, Probability conservation]
\label{p-probability-conservation} For each integer $t\ge 1$ we have $\sum_{x\in\mathbb{Z}}P(x,t)=1$.
\end{proposition}

Denote by $a_1(x,t)$ and $a_2(x,t)$ the real and imaginary part of %the vector
${a}(x,t)$ respectively.
%; see Figure~\ref{a2-graph}.
These values satisfy the following recurrence relations, the second one is also a tool for proving Proposition \ref{p-probability-conservation}.

\begin{proposition}[\cite{Preprint}, Proposition 7, Symmetry]
For each integer $x$ and each positive integer $t$ we have
$$a_{1}(x, t)=a_{1}(-x, t)\qquad\text{and}\qquad a_{2}(x, t)+a_{1}(x, t)=a_{2}(2-x, t)+a_{1}(2-x, t)$$
\end{proposition}
\begin{proposition}[\cite{Preprint}, Proposition 1, Dirac equation] \label{p-Dirac}  For each integer $x$ and each positive integer $t$ we have
\begin{align*}
a_1(x,t+1)&=\frac{1}{\sqrt{2}}a_2(x+1,t)+\frac{1}{\sqrt{2}}a_1(x+1,t);\\
a_2(x,t+1)&=\frac{1}{\sqrt{2}}a_2(x-1,t)-\frac{1}{\sqrt{2}}a_1(x-1,t).
\end{align*}
\end{proposition}

\subsection{Direction and spin}
%A feature of the model is that the electron's switch of direction emerges naturally rather than is added artificially.
A feature of the model is that the electron spin emerges naturally rather than is added artificially.

It goes almost without saying to consider the electron as being in one of the two states depending on the last-move direction: \emph{right-moving} or \emph{left-moving} (or just `\emph{right}' or `\emph{left}' for brevity).

The \emph{probability to find a right %(respectively, left)
electron in the square $(x,t)$, if a right electron was emitted from the square $(0,0)$}, is the length square of the vector $\sum_s {a}(s)$, where the sum is over only those paths from $(0,0)$ to $(x,t)$, which both start and finish with an upwards-right move. 

The \emph{probability to find a left electron} is defined analogously, only the sum is taken over paths which start with an upwards-right move but finish with an upwards-left move. Clearly, these probabilities equal $a_2(x,t)^2$ and $a_1(x,t)^2$ respectively, because the last move is directed upwards-right if and only if the number of turns is even. These right and left electrons are exactly the $(1+1)$-dimensional analogue of chirality states for a spin 1/2 particle.

We can also consider the probability $\sum\limits_{x\in\mathbb{Z}}a_{1}(x,t)^2$ to find a left electron at a time $t$, regardless of the coordinate $x$.

\begin{theorem}[A.Ustinov, \cite{Preprint}, Theorem 2, Probability of direction reversal] \label{p-right-prob}

For integer $t>0$ we have
%$\sum_{x\in\mathbb{Z}}a_{1}(x,t)^2=\frac{1}{2}\sum_{k=0}^{\lfloor t/2\rfloor-1}\frac{1}{(-4)^k}\binom{2k}{k}$. Thus We have
$$\sum_{x\in\mathbb{Z}}a_{1}(x,t)^2
=\frac{1}{2\sqrt{2}}+\mathrm{O}\left(\frac{1}{\sqrt{t}}\right).$$
%The generating function of the sequence is $\frac{t}{2(1-t)\sqrt{1+t^2}}$. WRONG???
\end{theorem}

\subsection{Lattice step and particle mass}
The basic model can be generalized by introducting additional parameters.

\begin{definition}[\cite{Preprint}, Definition 2] \label{def-mass}
Fix $\varepsilon>0$ and $m\ge 0$ called \emph{lattice step} and \emph{particle mass} respectively. Consider the lattice $\varepsilon\mathbb{Z}^2
=\{\,(x,t):x/\varepsilon,t/\varepsilon\in\mathbb{Z}\,\}$; see Figure~\ref{fig-limit}. \emph{Checker paths} $s$ on $\varepsilon\mathbb{Z}^2$ and their \emph{number of turns} $\mathrm{turns}(s)$ are defined analogously to those on $\mathbb{Z}^2$; see Definition~\ref{def-basic}.
For each $(x,t)\in\varepsilon\mathbb{Z}^2$, where $t>0$, denote by
$$
{a}(x,t,m,\varepsilon)
:=(1+m^2\varepsilon^2)^{(1-t/\varepsilon)/2}\,i\,\sum_s
(-im\varepsilon)^{\mathrm{turns}(s)}
$$
the sum over all checker paths $s$ on $\varepsilon\mathbb{Z}^2$ from $(0,0)$ to $(x,t)$ with the first step to $(\varepsilon,\varepsilon)$. Denote $$P(x,t,m,\varepsilon)
:=|{a}(x,t,m,\varepsilon)|^2.
$$
\begin{figure}[htbp]
  \centering
  \includegraphics[width=0.25\textwidth]{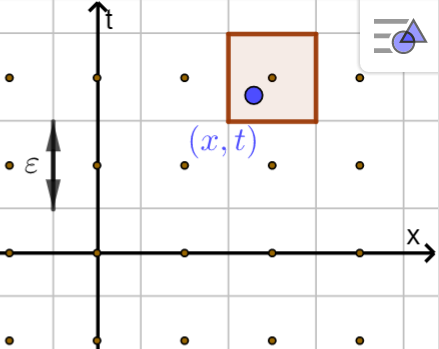}
  \hspace{0.5cm}
  \includegraphics[width=0.26\textwidth]{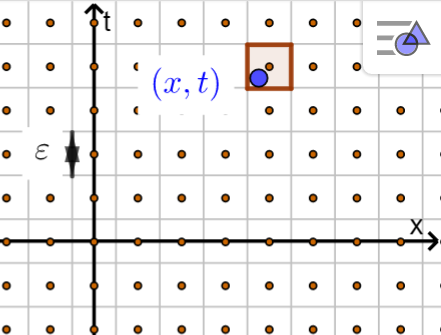}
  \hspace{0.5cm}
  \includegraphics[width=0.264\textwidth]{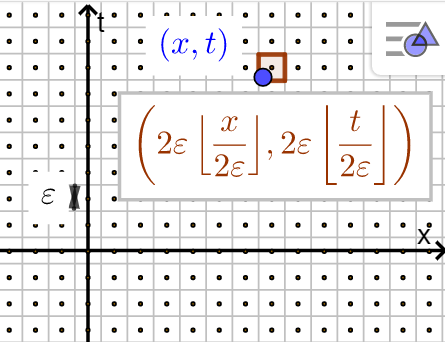}
  \caption{(by M.Skopenkov-A.Ustinov) The point stays fixed while the lattice step tends to zero.
  }\label{fig-limit}
\end{figure}
Denote by $a_1(x,t,m,\varepsilon)$ and $a_2(x,t,m,\varepsilon)$ the real and the imaginary part of $a(x,t,m,\varepsilon)$ respectively. In what follows we write $a_i(x, t) = a_i(x, t, m, \ee)$; all $a_i(x, t)$ depend on $m, \ee$ unless otherwise specified.
%For a function $a(\dots)$ denote by $a_1(\dots)$ and $a_2(\dots)$ the real and the imaginary part of $a$ respectively. In particular, \friday{$a(x,t,m,\varepsilon) =a_1(x,t,m,\varepsilon) +i\,a_2(x,t,m,\varepsilon)$.}
%Sometimes the vector $\vec a$ is considered as a complex number $a_1+ia_2$ (although complex numbers are not required for the solution of most problems). In what follows assume that $x$ and $t$ have the same parity unless the opposite is indicated.
\end{definition}
\begin{figure}
  \centering
  \includegraphics[width=0.6\textwidth]{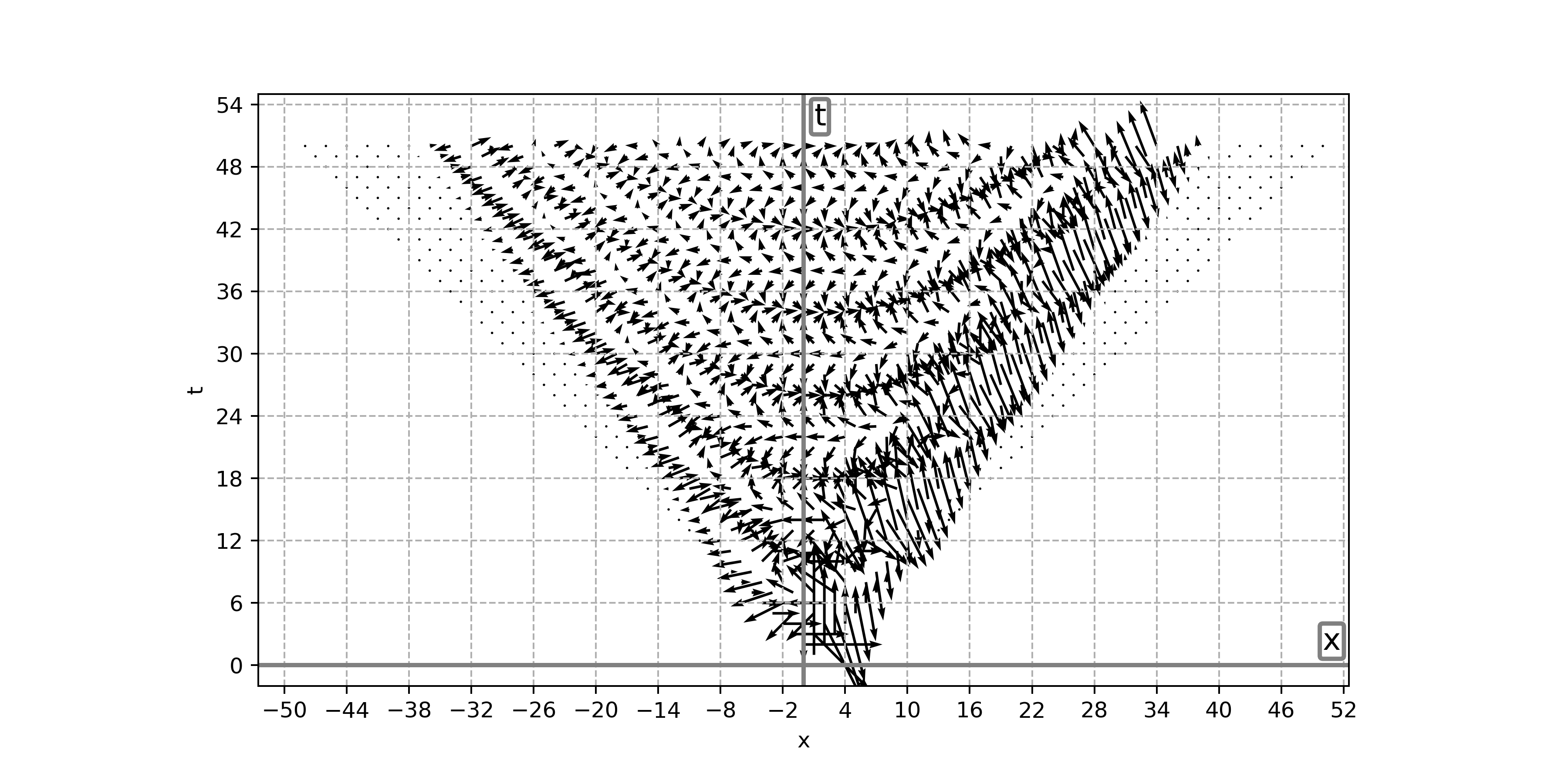}
  \caption{(by M.~Fedorov) The vectors $10\cdot a(x,t)$ for $t\le 50$}.
  \label{fig-arrows}
\end{figure}

For instance, $P(x,t,1,1)=P(x,t)$. 
One interprets $P(x,t,m,\varepsilon)$ as the probability to find an electron of mass $m$ in the square $\varepsilon\times\varepsilon$ with the center $(x,t)$, if the electron was emitted from the origin. Notice that the value $m\varepsilon$, hence $P(x,t,m,\varepsilon)$, is dimensionless in the natural units, where $\hbar=c=1$.

The following generalization of Theorem \ref{p-right-prob} is one of our main results; it is proved in the next section. 
\begin{theorem}[Probability of direction reversal] \label{new-theorem}
If $0\le m \varepsilon\le 1$ then
$$\lim_{\substack{t\to+\infty \\ t\in\varepsilon\mathbb{Z}}} \sum_{x\in\varepsilon\mathbb{Z}} a_1(x, t, m,\varepsilon)^2 = \frac{m\varepsilon}{2\sqrt{1+m^2\varepsilon^2}}.$$	
\end{theorem}
This confirms a conjecture by I.~Gaidai-Turlov--T.~Kovalev--A.~Lvov. The proof requires completely new ideas compared to Theorem 1, namely, application of Legendre polynomials and their asymptotic forms. 
For the first time Jacobi polynomials (with Legendre polynomials being a particular case) were applied to the Feynman checkers model in \cite[Lemma 5]{Jacobi_inventors}.

Besides, this theorem has a very
limited physical interpretation: in continuum theory the probability of direction reversal (for an electron emitted
by a point source) is ill-defined because the definition involves the square of the Dirac delta-function. A  more reasonable quantity related to direction is studied in \cite[p. 381]{Jacobi_by_physicists}.  

Each theorem below has also a simpler analogy in the base model.

\begin{proposition}[\cite{Preprint}, Proposition 4, Dirac equation]\label{p-mass}
For each $(x,t)\in \varepsilon\mathbb{Z}^2$, where $t>0$, we have
\begin{align}\label{eq-Dirac-mass1}
a_1(x,t+\varepsilon,m, \varepsilon) &= \frac{1}{\sqrt{1+m^2\varepsilon^2}}
(a_1(x+\varepsilon,t,m, \varepsilon)
+ m \varepsilon\, a_2(x+\varepsilon,t,m, \varepsilon)),\\
\label{eq-Dirac-mass2}
a_2(x,t+\varepsilon,m, \varepsilon) &= \frac{1}{\sqrt{1+m^2\varepsilon^2}}
(a_2(x-\varepsilon,t,m, \varepsilon)
- m \varepsilon\, a_1(x-\varepsilon,t,m, \varepsilon)).
\end{align}
\end{proposition}

\begin{proposition}[\cite{Preprint}, Proposition 5, Probability conservation] \label{p-mass2}
For each $t\!\in\!\varepsilon\mathbb{Z}$, $\!t>\!0$, we get $$\sum\limits
_{x\in\varepsilon\mathbb{Z}}P(x,t,m, \varepsilon)=1.$$
\end{proposition}

\begin{proposition}[\cite{Preprint}, Proposition 7, Symmetry]\label{p-symmetry}  
For each $(x, t) \in \varepsilon \mathbb{Z}^{2},$ where $t>0,$ we have
$$
\begin{array}{l}
a_{1}(x, t, m, \varepsilon)=a_{1}(-x, t, m, \varepsilon), \quad(t-x) a_{2}(x, t, m, \varepsilon)=(t+x-2 \varepsilon) a_{2}(2 \varepsilon-x, t, m, \varepsilon) \\
\quad a_{1}(x, t, m, \varepsilon)+m \varepsilon a_{2}(x, t, m, \varepsilon)=a_{1}(2 \varepsilon-x, t, m, \varepsilon)+m \varepsilon a_{2}(2 \varepsilon-x, t, m, \varepsilon)
\end{array}
$$
\end{proposition}

\begin{proposition}[\cite{Preprint}, Proposition 9, ``Explicit'' formula] \label{p-mass3}
For each integers $t>|x|$ such that $x+t$ is even we have
\begin{align}
a_1(x\varepsilon,t\varepsilon,m,\varepsilon) &=
(1+m^2\varepsilon^2)^{(1-t)/2}\sum_{r=0}^{(t-|x|)/2}(-1)^r \binom{(x+t-2)/2}{r}\binom{(t-x-2)/2}{r}(m\varepsilon)^{2r+1},
\label{eq1-p-mass}\\
a_2(x\varepsilon,t\varepsilon,m,\varepsilon)&=
(1+m^2\varepsilon^2)^{(1-t)/2}\sum_{r=1}^{(t-|x|)/2}(-1)^r \binom{(x+t-2)/2}{r}\binom{(t-x-2)/2}{r-1}(m\varepsilon)^{2r};
\label{eq2-p-mass}%\\
%\intertext{For each $t=x>0$ we have}
%a_1(t,t,m\varepsilon)&=0,\\
%a_2(t,t,m\varepsilon)&=(1+m^2\varepsilon^2)^{(1-t)/2};\\
%\intertext{For each $0<t<|x|$ or $t=-x>0$ we have}
%a_1(x,t,m\varepsilon)&=0,\\
%a_2(x,t,m\varepsilon)&=0.
\end{align}
\end{proposition}

\begin{proposition}[\cite{Preprint}, Proposition 4, Huygens' principle]\label{p-Huygens} For each $x,t,t'\in\varepsilon\mathbb{Z}$, where $t>t'>0$, we have
\begin{align*}
a_1(x,t,m, \varepsilon)  &=\sum \limits_{x'\in\varepsilon\mathbb{Z}} \left[ a_2(x',t')a_1(x-x'+\ee,t-t'+\ee) + a_1(x',t')a_2(x'-x+\ee,t-t'+\ee) \right],\\
a_2(x,t,m, \varepsilon)  &= \sum \limits_{x'\in\varepsilon\mathbb{Z}} \left[ a_2(x',t')a_2(x-x'+\ee,t-t'+\ee) - a_1(x',t')a_1(x'-x+\ee,t-t'+\ee) \right].
\end{align*}
\end{proposition}

\section{The proof of the main theorem}
In this section all $a_i(x,t)$ depend on $m\ee$, i.e. $a_i(x, t) = a_i(x, t, m, \ee)$.

\begin{proof}[Proof of the Theorem~\ref{new-theorem} modulo some lemmas]
The theorem follows from the sequence of computations explained in the lemmas below.

Consider the Legendre polynomials $P_n(x) = \frac{1}{2^{n}n!} \frac{d^n}{dx^n}(x^2 - 1)^n.$  Then
\begin{align*}
\lim_{\substack{t\to+\infty \\ t\in\varepsilon\mathbb{Z}}} \sum\limits_{x\in\varepsilon\mathbb{Z}}a_1(x, t, m, \varepsilon)^2 &\stackrel{(1)}{=} \frac{m\varepsilon}{\sqrt{1 + m^2\varepsilon^2}}\sum\limits^{\infty}_{n=0}a_1(0, (2n + 2)\varepsilon)\\
\\ &\stackrel{(2)}{=}  \frac{m^2\ee^2}{1 + m^2\ee^2} \sum\limits_{n=0}^{\infty} P_n\Bigg(\frac{1 - m^2\ee^2}{1 + m^2\ee^2}\Bigg) \stackrel{(3)}{=} \frac{m^2\ee^2}{1 + m^2\ee^2} \frac{1}{\sqrt{2 - 2\frac{1 - m^2\ee^2}{1 + m^2\ee^2}}} = \frac{m\ee}{2\sqrt{1 + m^2\ee^2}}.
\end{align*} 
Here (1) -- (3) follow from Lemmas \ref{recurrent-lemma} -- \ref{lemma-juan} respectively.
\end{proof}

The first lemma is essentially taken from \cite[Proof of Theorem 5]{Preprint}.
\begin{lemma}\label{recurrent-lemma}
For each $t \in \ee\mathbb{Z}$, where $t > 0$
$$\sum\limits_{x  \in \ee\mathbb{Z}}a_1^2(x, t) = \frac{m\varepsilon}{\sqrt{1 + m^2\varepsilon^2}}\sum\limits^{t/\ee-1}_{k=1}a_1(0, 2k\ee)$$
\end{lemma}
\begin{proof}
Set $S_i(t) = \sum\limits_{x \in \ee\mathbb{Z}}a_i^2(x, t)$ for $i = 1, 2$ and $S_{12}(t) =\sum\limits_{x \in \ee\mathbb{Z}}a_1(x, t)a_2(x, t)$. It suffices to prove that
$$S_1(t + \ee) = S_1(t) + \frac{m\varepsilon}{\sqrt{1 + m^2\varepsilon^2}}a_1(0, 2t).$$
Decompose $a_1(0, 2t)$ using Proposition \ref{p-Huygens} for $t' = t, x=0$ and then the Dirac equation (Proposition \ref{eq-Dirac-mass1})
\begin{align*}
 a_1(0, 2t) &=
\sum \limits_{x' \in \ee \mathbb{Z}} \left[ a_2(x',t)a_1(x' - \ee,t + \ee) + a_1(x',t)a_2(x'+ \ee,t+\ee) \right] \\
&= \frac{1}{\sqrt{1 + m^2\varepsilon^2}}\sum \limits_{x' \in \ee\mathbb{Z}}\Big[ a_2(x', t)\Big(a_1(x', t) + m\varepsilon a_2(x', t)\Big) + a_1(x', t)\Big(-m\varepsilon a_1(x', t) + a_2(x', t)\Big)\Big] \\
&= \frac{1}{\sqrt{1 + m^2\varepsilon^2}}\Big[2S_{12}(t) + m\varepsilon \, S_2(t) - m\varepsilon \, S_1(t) \Big].\\
\end{align*}
Thus, 
\begin{equation} \label{some-tricky-relation}
2S_{12}(t) + m\varepsilon S_2(t) = \sqrt{1 + m^2\varepsilon^2} a_1(0, 2t) + m\varepsilon S_1(t).
\end{equation}
On the other hand,
\begin{multline*}
S_1(t + \ee) - S_2(t + \ee) = \sum \limits_{x}\Big( a^2_1(x, t+\ee, m,\varepsilon) - a^2_2(x, t+\ee, m, \varepsilon)\Big) = \\ 
 =
\sum \limits_{x}\Bigg[\Bigg(\frac{1}{\sqrt{1 + m^2\varepsilon^2}}\Big(a_1(x + \ee, t) + m\varepsilon\, a_2(x + \ee, t)\Big)\Bigg)^2 - \Bigg(\frac{1}{\sqrt{1 + m^2\varepsilon^2}}\Big(-m\varepsilon \, a_1(x - \ee, t) + a_2(x - \ee, t)\Big)\Bigg)^2 \Bigg] \\
= \frac{1}{1 + m^2\varepsilon^2}\Bigg((1 - m^2\varepsilon^2)(S_1(t) - S_2(t)) + 4m\varepsilon\, S_{12}(t)\Bigg).
\end{multline*}
Again, we get the second equality using the Dirac equation (Proposition \ref{eq-Dirac-mass1}). After adding the equality $S_1(t + \ee) + S_2(t + \ee) = S_1(t) + S_2(t)$ (the equality holds because both sides equal to $1$) to the previous one we get:
\begin{align*}
2 S_1(t + 1) &= \frac{2S_1(t) + 2m\varepsilon\big(2S_{12}(t) + m\varepsilon S_2(t)\big)}{1 + m^2 \varepsilon^2} \\
&= \frac{2S_1(t) + 2m\varepsilon\big(\sqrt{1 + m^2\varepsilon^2} a_1(0, 2t) + m\varepsilon S_1(t)\big)}{1 + m^2 \varepsilon^2} && \text{ (by \eqref{some-tricky-relation})}.
\end{align*}
We obtain the required equality:
\begin{equation}
S_1(t + \ee) = S_1(t) + \frac{m\varepsilon}{\sqrt{1 + m^2\varepsilon^2}}a_1(0, 2t).
\end{equation}
Thus,
$$\sum\limits_{x \in \ee\mathbb{Z}}a_1^2(x, t) = S_1(t) = \frac{m\varepsilon}{\sqrt{1 + m^2\varepsilon^2}}\sum\limits^{t/\ee-1}_{k=1}a_1(0, 2k).$$
\end{proof}

%\begin{theorem} \label{my-theorem}
%$$\lim_{t \to \infty} \sum\limits_{x \in \mathbb{Z}}a_1(x, t, m, \varepsilon)^2 = \frac{m\varepsilon}{2\sqrt{1 + m^2 \varepsilon^2}}.$$
%\end{theorem}

The second lemma is a particular case of \cite[Lemma 5]{Jacobi_inventors}.
\begin{lemma}[]\label{to-legendre}
For each positive integer $n$
$$a_1(0, (2n + 2)\ee) = \frac{m\varepsilon}{\sqrt{1 + m^2\varepsilon^2}} P_n\Bigg(\frac{1 - m^2\ee^2}{1 + m^2\ee^2}\Bigg).$$
\end{lemma}
\begin{proof}
This follows from the sequence of computations:
\begin{align*}
a_1(0, (2n + 2)\ee)&\stackrel{(1)}{=} (1 + m^2\ee^2)^{1/2 - n  - 1} \sum\limits_{r=0}^{n+1}(-1)^r \binom{n}{r}^2(m\ee)^{2r+1} \\ &= \frac{m\ee}{\sqrt{1 + m^2\ee^2}}(1 + m^2\ee^2)^{-n}\sum\limits_{r=0}^{n}\binom{n}{r}^2(- m^2\ee^2)^{r} \\ &\stackrel{(2)}{=} \frac{m\ee}{\sqrt{1 + m^2\ee^2}} (1 + m^2\ee^2)^{-n}  (1 + m^2\ee^2)^n P_n\Bigg(\frac{1 - m^2\ee^2}{1 + m^2\ee^2}\Bigg) \\ &=  \frac{m\ee}{\sqrt{1 + m^2\ee^2}} P_n\Bigg(\frac{1 - m^2\ee^2}{1 + m^2\ee^2}\Bigg).
\end{align*}
Here (1) holds by Proposition \ref{p-mass3} and equality $(2)$ for $x \ne 1$ follows from \cite[(4.2.7.6)]{integrals-and-series}:
$$\sum\limits_{r=0}^{n}\binom{n}{r}^2x^{r} = (1 - x)^n P_n\Bigg(\frac{1 + x}{1 - x}\Bigg).$$
\end{proof}

The third lemma is the Fatou theorem applied to the generating function for the Legendre polynomials.
\begin{lemma}{} \label{lemma-juan}
For each $0<x<1$ we have
$$\sum\limits_{n=0}^{\infty} P_n(x) = \frac{1}{\sqrt{2 - 2x}}.$$
\end{lemma}
\begin{proof}
The Legendre polynomials are the coefficients  in the formal Taylor series (see \cite[IV.1]{Suetin}):
$$\frac{1}{\sqrt{1 - 2xt + t^2}} = \sum\limits_{n=0}^{\infty} P_n(x)t^n.$$ 

The coefficients satisfy the \emph{Tauberian condition} $P_n(x) \to 0$ as $n \to \infty$ for each fixed ${0 \le x\le 1}$ by \cite[(4.6.7)]{Lebedev} and the left-hand side has an analytic continuation to  a neighborhood of the point $t=1$. Then by the Fatou theorem \cite[Theorem 12.1]{Korevaar},
the Taylor series  converge to the value of the function, that is, the assertion of the lemma holds.
\end{proof}

\section{An asymptotic form for $a_1(x, t, m, \ee)$ \\ in the particular case $x = 0$}
Let us consider the Feynman problem in the particular case ${x = 0}$ using the expression of the function $a_1(0, t, m, \ee)$ through Legendre polynomials from the previous section.  

Now we prove the following result.
\begin{theorem}  \label{asymptote}
For each $\delta, m, \ee > 0$ such that $\delta \le m\ee \le 1 - \delta$ and each integer $n$ we have 
$$a_1(0, (2n + 2)\ee, m, \ee) = \sqrt{\frac{m\ee}{2\pi n}} \cos \left(\left( 2n+1 \right)\arctan(m\ee)-\frac{\pi}{4} \right) + O_{\delta}\left(n^{-3 / 2}\right).$$

Hereafter we write $f(n, m, \ee) = O_\delta(g(n))$ for $\delta > 0$  if there is positive $C$ such that for each
$n, m, \ee$ satisfying the assumptions of the theorem we have 
$$|f(m, n, \ee)| < C \cdot g(n)$$ 
\end{theorem} 
\begin{proof}
Consider the asymptotic form for the Legendre polynomial \cite[p.194, Theorem 8.21.2]{Gabor}:
$$P_{n}(\cos \theta)=2^{1 / 2}(\pi n \sin \theta)^{-1 / 2} \cos \left\{\left(n+\frac{1}{2}\right) \theta-\pi / 4\right\}+O_{\eta}\left(n^{-3 / 2}\right),$$
where $\eta \le \theta \le \pi - \eta$ .
Since $0 \le m\ee \le 1$, there is $\theta$ such that $\cos(\theta) = (1 - m^2\ee^2)/(1 + m^2\ee^2)$. Clearly, then $\tan \big(\theta / 2\big) = m\ee$ and $\sin(\theta) = (2m\ee)/(1 + m^2\ee^2)$ and the proposition follows from Lemma \ref{to-legendre}. 
\end{proof}
\begin{remark} Note that the restriction $0 < \delta \le m\ee$ is essential here, as one can see from \\ \cite[Theorem~3]{Preprint}.
\end{remark}

\section{Model with External field}
Consider an infinite checkerboard with the centers of the squares at the integer points.
An electromagnetic field is viewed as a fixed assignment $u$ of numbers $+1$ and $-1$ to
all the vertices of the squares.
In this model we modify the definition of the vector $a(s)$ by
reversing the direction each time when the checker passes through a vertex with the field $-1$. Denote by $a(s, u)$ the resulting vector. Define $a(x, t, u)$ and $P(x, t, u)$ analogously to $a(x, t)$ and $P(x, t)$ replacing $a(s)$ by $a(s, u)$ in the definition. For instance, if $u = +1$ identically, then $P(x, t, u) = P(x, t)$. Again, we summarize this construction rigorously.

\begin{definition}{\cite[Definition 3]{Preprint}}
An \textit{edge} is a segment joining nearest-neighbor integer points with even sum of the coordinates. Let $u$ be a map from the set of all edges to $\{-1, 1\}$. Denote by
$$
a(x, t, u):=2^{(1-t) / 2} i \sum_{s}(-i)^{\text {turns }(s)} u\left(s_{0} s_{1}\right) u\left(s_{1} s_{2}\right) \ldots u\left(s_{t-1} s_{t}\right)
$$
the sum over all checker paths $s=\left(s_{0}, s_{1}, \ldots, s_{t}\right)$ with $s_0 = (0, 0)$, $s_1 = (1, 1)$,
and $s_t = (x, t)$. Set $P(x, t, u) := |a(x, t, u)|^2$.
Denote by $a_1(x, t, u)$ and $a_2(x, t, u)$ the real and the imaginary part of $a(x, t, u)$. For half-integers $x, t$
denote by $u(x, t)$ the value of $u$ on the edge with the midpoint $(x, t)$. 
\end{definition}
\begin{figure}
  \centering
  \begin{tabular}{@{}c@{}}
    \includegraphics[width=.7\linewidth,height=100pt]{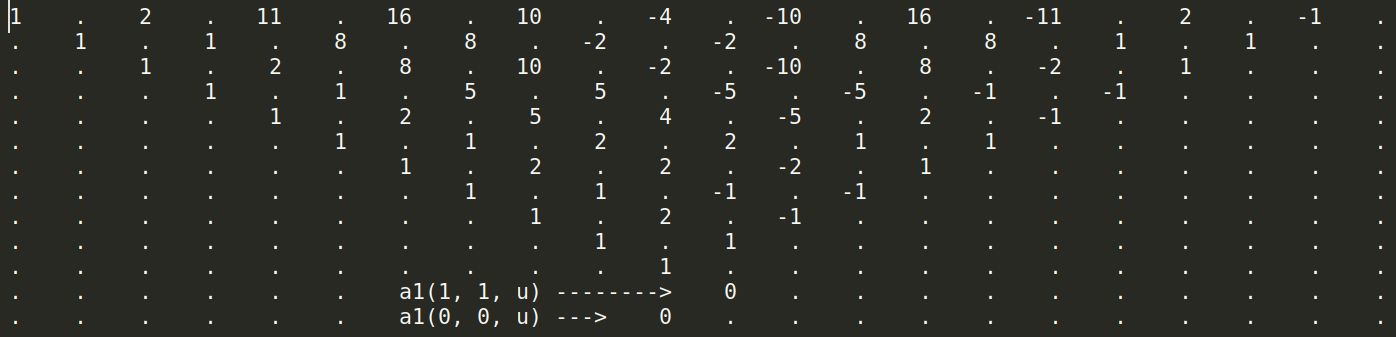} \\[\abovecaptionskip]
    \small $2^{(t - 1)/2}\cdot a_1(x, t, u)$
  \end{tabular}

  \vspace{\floatsep}

  \begin{tabular}{@{}c@{}}
    \includegraphics[width=.7\linewidth,height=100pt]{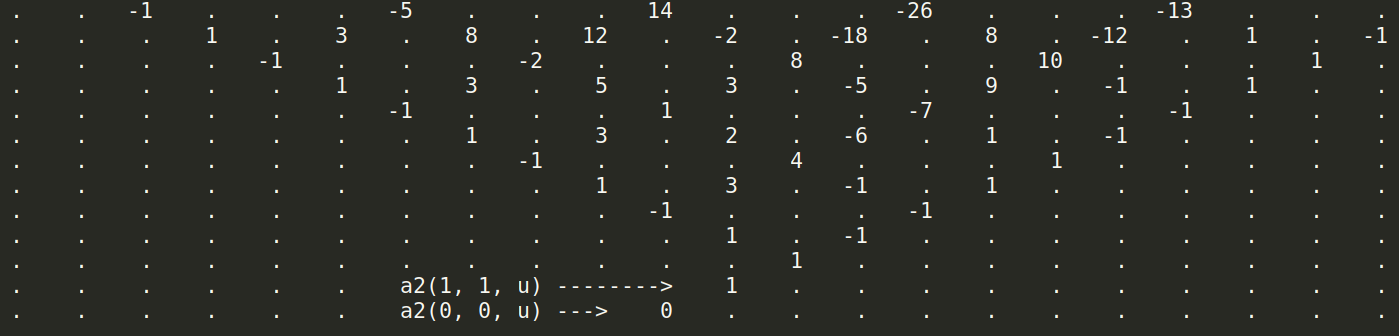} \\[\abovecaptionskip]
    \small $2^{(t - 1)/2}\cdot a_2(x, t, u)$
  \end{tabular}

  \caption{The values $2^{(t - 1)/2} \cdot a(x, t, u)$}\label{fig:a1a2}
\end{figure}
In Figure \ref{fig:a1a2} the values $a_1(x,t,u)$ and $a_2(x,t,u)$ are shown for small $x, t$ and the electromagnetic field $u$ given by $u(x + 1/2, t + 1/2) = -1$, if both $x$ and
$t$ even, and $u(x + 1/2, t + 1/2) = +1$ otherwise (this is called \textit{homogenious magnetic field}).

\begin{proposition}[Probability/charge conservation]\cite[Proposition 15, ]{Preprint}
\textit{For each integer} $x$ \textit{and} $t \ge 1$ we have $\sum_{x \in \mathbb{Z}} P(x, t, u)=1$.
\end{proposition}

\begin{proposition}[Dirac equation in electromagnetic field]\cite[Proposition 14]{Preprint}\label{dirac-equation-in-electromagnetic-field}
\textit{For each integers $x$ and $t \ge 1$ we have}
\begin{equation*}
\begin{split}
a_{1}(x, t+1, u)=\frac{1}{\sqrt{2}} u\left(x+\frac{1}{2}, t+\frac{1}{2}\right)\left(a_{1}(x+1, t, u)+a_{2}(x+1, t, u)\right), \\
a_{2}(x, t+1, u)=\frac{1}{\sqrt{2}} u\left(x-\frac{1}{2}, t+\frac{1}{2}\right)\left(a_{2}(x-1, t, u)-a_{1}(x-1, t, u)\right) .
\end{split}
\end{equation*}
\end{proposition}

\subsection{Linear relations in quadruples}
If we take a closer look to Figure~\ref{fig:a1a2}, we will see that numbers $a_1(x, t, u)$ and $a_2(x, t, u)$ can be divided into diamond-shaped quadruples with simple linear relations between the members of each quadruple. Such quadruples are called \textit{diamonds} because of their shape.
\begin{theorem} \label{diamond-theorem}
Let $u\left(x+\frac{1}{2}, t+\frac{1}{2}\right)=(-1)^{(x-1)(t-1)}$ for all $(x, t) \in \mathbb{Z}^2$. For each $(x, t) \in \mathbb{Z}^2$ such that $t \ge 1$ and either $(x, t) \underset{4}{\equiv} (2, 1)$ or $(x, t) \underset{4}{\equiv} (0, 3)$ the following equalities hold:
\begin{enumerate}
\item 
$\sqrt{2} \cdot a_1(x + 1, t, u) = a_1(x, t - 1, u) = \sqrt{2} \cdot a_1(x - 1, t, u) = a_1(x, t + 1, u) = \\ \sqrt{2} \cdot a_2(x + 1, t, u) = \sqrt{2} \cdot a_2(x - 1, t, u) - 2 a_2(x, t + 1, u);$

\item 
$a_2(x, t - 1, u) = 0.$
\end{enumerate}
\end{theorem}

\begin{example} \label{Example1}
For $(x, t) = (0, 3)$ we have \\
\begin{equation} \label{eq1}
\begin{split}{\notag}
\frac{1}{\sqrt{2}} & = \sqrt{2} \cdot a_1(1, 3, u) = a_1(0, 2, u) = \sqrt{2} \cdot a_1(-1, 3, u) = a_1(0, 4, u) = \\
& = \sqrt{2} \cdot a_2(1, 3, u) = \sqrt{2} \cdot a_2(-1, 3, u) - 2 \cdot a_2(0, 4, u); \\
0 & = a_2(0, 2, u) .
\end{split}
\end{equation}
For $(x, t) = (2, 1)$ we have \\
\begin{equation} \label{eq1}
\begin{split}{\notag}
0 & = \sqrt{2} \cdot a_1(3, 1, u) = a_1(2, 0, u) = \sqrt{2} \cdot a_1(1, 1, u) = a_1(2, 2, u) = \\
& = \sqrt{2} \cdot a_2(3, 1, u) = \sqrt{2} \cdot a_2(1, 1, u) - 2 \cdot a_2(2, 2, u); \\
0 & = a_2(2, 0, u).
\end{split}
\end{equation}
\end{example}

\begin{proof}[Proof of Theorem \ref{diamond-theorem}]

We prove this by induction on $t$ using Proposition \ref{dirac-equation-in-electromagnetic-field}.\\
Induction base ($t = 1$) is Example \ref{Example1}. Induction step consists of 6 steps:

\begin{enumerate}
\item 
$\sqrt{2} \cdot a_1(x + 1, t, u) \stackrel{(Step 1)}{=} a_1(x, t - 1, u) \stackrel{(Step 3)}{=} \sqrt{2} \cdot a_1(x - 1, t, u) \stackrel{(*)}{=} a_1(x, t + 1, u) \stackrel{(Step 5)}{=} \\ \sqrt{2} \cdot a_2(x + 1, t, u) \stackrel{(Step 6)}{=} \sqrt{2} \cdot a_2(x - 1, t, u) - 2 a_2(x, t + 1, u);$

\item 
$a_2(x, t - 1, u) \stackrel{(Step 2)}{=} 0;$
\end{enumerate}

\textit{Step 4}: $a_1(x, t - 1, u) = \sqrt{2} \cdot a_2(x + 1, t, u).$	
Then $(*)$ is going to follow from Steps 3-5. \\

Dirac equation in electromagnetic field (Proposition \ref{dirac-equation-in-electromagnetic-field}) gives the following relations:

\begin{equation}
\label{third_bracket}
\begin{cases}
      \sqrt{2}  \cdot a_1(x, t + 1, u) = a_1(x + 1, t, u) + a_2(x + 1, t, u), \\ 
      \sqrt{2}  \cdot a_1(x - 1, t, u) = a_1(x, t - 1, u)  + a_2(x, t - 1, u), \\
      \sqrt{2}  \cdot a_2(x + 1, t, u) = - (a_2(x, t - 1, u) - a_1(x, t - 1, u) ), \\
      \sqrt{2}  \cdot  a_2(x, t + 1, u) = a_2(x - 1, t, u) - a_1(x - 1, t, u), \\
\end{cases}  
\end{equation}
\begin{equation}
\label{fourth_bracket}
\begin{cases}
      \sqrt{2}  \cdot  a_1(x + 1, t, u) = a_1(x + 2, t - 1, u) + a_2(x + 2, t - 1, u), \\
      \sqrt{2}  \cdot  a_1(x, t - 1, u) = a_1(x + 1, t - 2, u) + a_2(x + 1, t - 2, u), \\
      \sqrt{2}  \cdot  a_2(x - 1, t, u) = - (a_2(x - 2, t - 1, u)- a_1(x - 2, t - 1, u)), \\
      \sqrt{2}  \cdot  a_2(x, t - 1, u) = a_2(x - 1, t - 2, u) - a_1(x - 1, t - 2, u). \\
\end{cases}
\end{equation} 
Since $(x, t) \underset{4}{\equiv} (2, 1)$ or $(x, t) \underset{4}{\equiv} (0, 3)$ (the latter being equivalent to $(x, t - 1) \underset{2}{\equiv} (0, 0)$), there is a minus sign in the third equation of (\ref{third_bracket}).
Since $(x, t) \underset{4}{\equiv} (2, 1)$ or $(x, t) \underset{4}{\equiv} (0, 3)$ (the latter being equivalent to $(x - 2, t - 1) \underset{2}{\equiv} (0, 0)$), there is a minus sign in the third equation of (\ref{fourth_bracket}).

\textit{Step 1.} From (\ref{fourth_bracket}) we have $\sqrt{2}  \cdot  a_1(x + 1, t, u) = a_1(x + 2, t - 1, u) + a_2(x + 2, t - 1, u)$. Taking the pair $(x + 2, t - 2)$ instead of $(x, t)$ and applying $(*)$ for it, by the induction hypothesis we get 
\begin{equation} \label{step-4-hypothesis-x+2}
a_1(x + 2, t - 1, u)  = \sqrt{2}  \cdot   a_1(x + 1, t - 2, u).
\end{equation} Hence, 
\begin{equation} \label{blue}
\sqrt{2}  \cdot  a_1(x + 1, t, u) = \sqrt{2}  \cdot   a_1(x + 1, t - 2, u) + a_2(x + 2, t - 1, u).
\end{equation}
On the other hand, from (\ref{fourth_bracket}) we have 
\begin{equation} \label{on-the-other-hand}
\sqrt{2}~\cdot~a_1(x, t - 1, u)~=~a_1(x + 1, t - 2, u)~+~a_2(x + 1, t - 2, u).
\end{equation}
Taking the pair $(x + 2, t - 2)$ instead of $(x, t)$, applying \textit{Step 6} for it and transforming the resulting equality, by the induction hypothesis we get 
\begin{equation}\label{step-6-hypothesis-x+2}
a_2(x + 1, t - 2, u) = \sqrt{2} \cdot a_2(x + 2, t - 1, u) + a_2(x + 3, t - 2, u).
\end{equation} 
By $(*)$ and \textit{Step 5} for $(x + 2, t - 2)$ we get 
\begin{equation}\label{step-4-step5-hypothesis-x+2}
a_2(x + 3, t - 2, u) = a_1(x + 1, t - 2, u).
\end{equation}
From (\ref{on-the-other-hand}), (\ref{step-6-hypothesis-x+2}) and (\ref{step-4-step5-hypothesis-x+2}) we get $\sqrt{2} \cdot a_1(x, t - 1, u) = 2 \cdot a_1(x + 1, t - 2, u) + \sqrt{2} \cdot a_2(x + 2, t - 1, u)$, which means
\begin{equation} \label{step-1-punchline}
a_1(x, t - 1, u) = \sqrt{2} \cdot a_1(x + 1, t - 2, u) + a_2(x + 2, t - 1, u).
\end{equation}
From (\ref{blue}) and (\ref{step-1-punchline}) we get
$\sqrt{2} \cdot a_1(x + 1, t, u) = a_1(x, t - 1, u).$\\

\textit{Step 2.} By (\ref{fourth_bracket}) we have $\sqrt{2} \cdot a_2(x, t - 1, u) = a_1(x - 1, t - 2, u) - a_2(x - 1, t - 2, u)$.
Taking the pair $(x - 2, t - 2)$ instead of $(x, t)$, by the induction hypothesis from \textit{Step 1} and \textit{Step 4} we  get $a_1(x - 1, t - 2, u) = a_2(x - 1, t - 2, u) $. This implies that $\sqrt{2} \cdot a_2(x, t - 1, u) = a_1(x - 1, t - 2, u) - a_1(x - 1, t - 2, u) = 0$.\\

\textit{Step 3.} By (\ref{third_bracket}) we have  $\sqrt{2} \cdot a_1(x - 1, t, u) = a_1(x, t - 1, u) + a_2(x, t - 1, u)$. By \textit{Step 2}, $a_2(x, t - 1, u) = 0$. Thus, $\sqrt{2} \cdot a_1(x - 1, t, u) = a_1(x, t - 1, u)$.\\

\textit{Step 4.} By (\ref{third_bracket}) we have $\sqrt{2} \cdot a_2(x + 1, t, u) = - (a_2(x, t - 1, u) - a_1(x, t - 1, u)$. By \textit{Step 2}, $a_2(x, t - 1, u) = 0$. Thus, $\sqrt{2} \cdot a_2(x + 1, t, u) = a_1(x, t - 1, u)$.\\

\textit{Step 5.} By (\ref{third_bracket}) we have $\sqrt{2} \cdot a_1(x, t + 1, u) = a_1(x + 1, t, u) + a_2(x + 1, t, u)$. \textit{Step 4} and \textit{Step 1} together give $a_1(x + 1, t, u) = a_2(x + 1, t, u)$. Thus, $a_1(x, t + 1, u) = \sqrt{2} \cdot a_2(x + 1, t, u)	$.\\

\textit{Step 6.} By (\ref{third_bracket}) we have $\sqrt{2} \cdot a_2(x, t + 1, u) = a_2(x - 1, t, u) - a_1(x - 1, t, u)$. \textit{Step 3} and \textit{Step 4} together give $a_1(x - 1, t, u) = a_2(x + 1, t, u)$, therefore, $\sqrt{2} \cdot a_2(x, t + 1, u) = a_2(x - 1, t, u) - a_2(x + 1, t, u)$. Thus, $a_2(x + 1, t, u)	 = a_2(x - 1, t, u) - \sqrt{2} \cdot a_2(x, t + 1, u)$.
\end{proof}

\subsection{Probability of direction reversal}

We are interested in the limit probability of direction reversal for the model with the homogeneous external magnetic field. While in the basic model such limit exists (Theorem \ref{p-right-prob} by A.Ustinov), numerical experiments show that the model with external field exhibits two limit points (see Figure \ref{fig:p_left}).

\begin{conjecture} \label{conjecture-probability-of-direction-reversal-external-field}
Let $u(x + 1/2, t + 1/2) = -1$, if both $x$ and
$t$ even, and $u(x + 1/2, t + 1/2) = +1$, otherwise. Then

\begin{align*}
\lim\limits_{t \to \infty} \sum\limits_{x \in \mathbb{Z}} a_1^2(x, 2t + 1, u) &= \frac{\sqrt{3}}{6},\\
\lim\limits_{t \to \infty} \sum\limits_{x \in \mathbb{Z}} a_1^2(x, 2t, u) &= \frac{\sqrt{3}}{3}.
\end{align*}
\end{conjecture}

\begin{figure}[htb]
  \centering
  \includegraphics[width=0.8\linewidth,height=300pt]{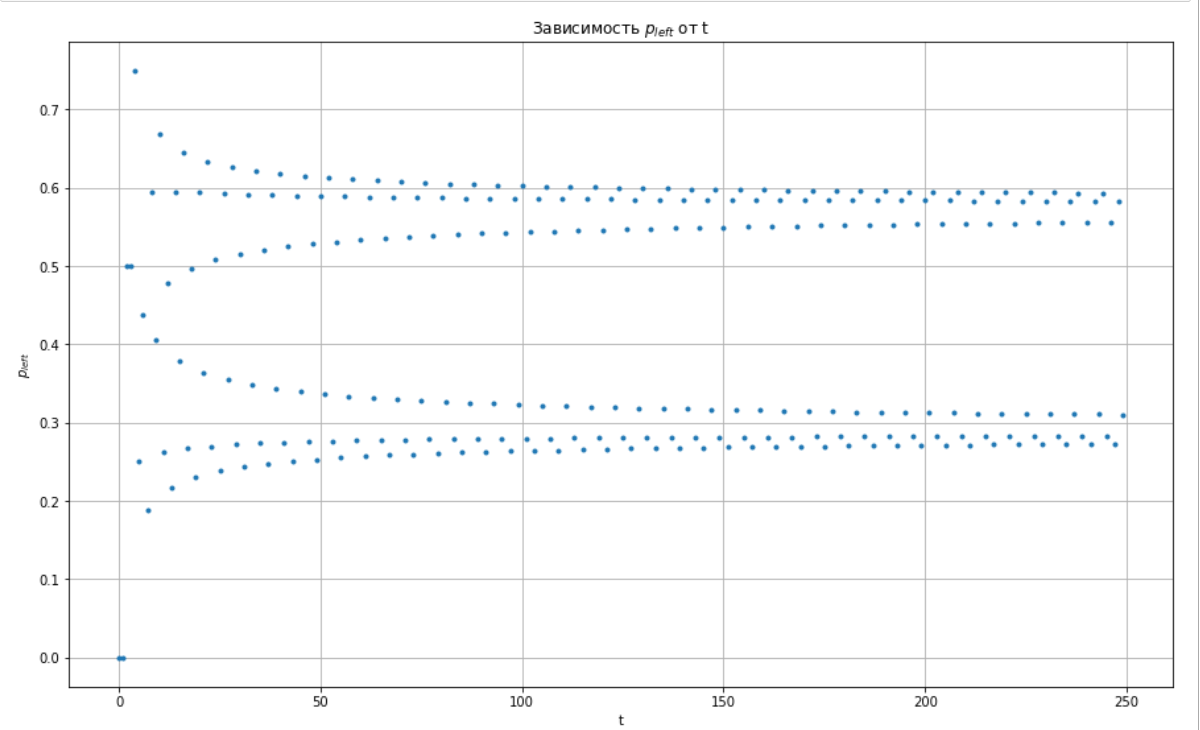}
  \caption{Plot of $p_{left}^t = \sum\limits_{x \in \mathbb{Z}} a_{1}^{2}(x, t, u)$.}
  \label{fig:p_left}
\end{figure}

\subsection{The new lattice}

\begin{definition}
For each $x, t \in \mathbb{Z}$ such that $x + t$ is even and $t>0$ define
\begin{equation*}
\label{first_bracket}
\begin{cases}
b_1(x, t) = a_1(2x - 1, 2t - 1, u)\\
b_2(x, t) = a_2(2x - 1, 2t - 1, u) \\
\end{cases}  
\end{equation*}
where  $u(x + 1/2, t + 1/2) = -1$, if both $x$ and
$t$ are even, and $u(x + 1/2, t + 1/2) = +1$ otherwise. For pairs $(x, t)$ such that $x + t$  is odd set $b_1(x, t) = b_2(x, t) = 0$.
\end{definition}

\begin{figure}[htb]
  \includegraphics[width=0.8\linewidth,height=300pt]{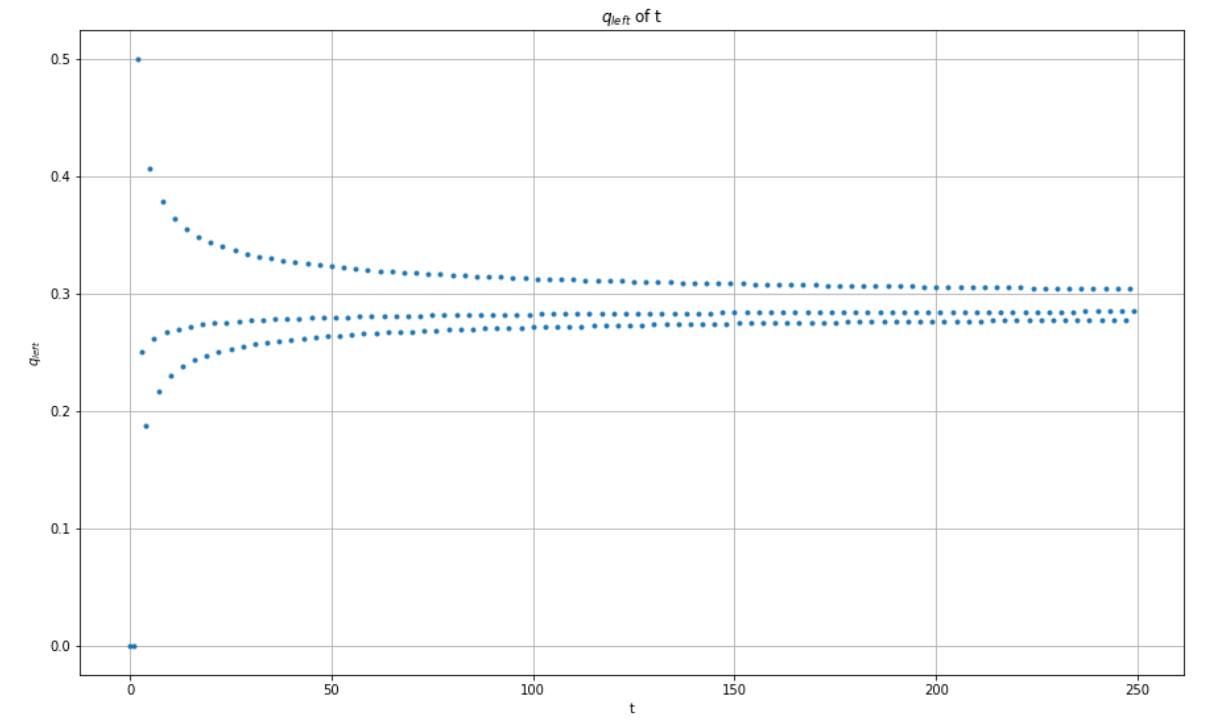}
  \caption{Plot for $q_{left}^t = \sum\limits_{x \in \mathbb{Z}} b_1^2(x, t).$}
  \label{fig:q_left}
\end{figure}

\begin{proposition} \label{proposition-new-lattice}
For each $x, t \in \mathbb{Z}$ such that $x + t$ is even and $t>1$ we have
\begin{align}
b_1(x, t) &= \frac{1}{2}\Big[b_1(x + 1, t - 1) + b_2(x + 1, t - 1)\Big], \label{new-lattice-recurrent-1} \\
b_2(x, t) &= \frac{1}{2}\Big[3 \cdot b_1(x - 1, t - 1) - b_2(x - 1, t -  1)\Big]. \label{new-lattice-recurrent-2}
\end{align}
\end{proposition}
\begin{proof}

Proof of (\ref{new-lattice-recurrent-1}). Let us rewrite the equation in terms of $a_1$ and $a_2$:
$$a_1(2x - 1, 2t - 1, u) = \frac{1}{2}\Big[a_1(2x + 1, 2t - 3, u) + a_2(2x + 1, 2t - 3, u)\Big].$$
Since $x + t$ is even, it follows that $(2x - 1, 2t - 1) \underset{4}{\equiv} (1, 1)$ or $(2x - 1, 2t - 1) \underset{4}{\equiv} (3, 3)$ and $u(2x-1/2,2t-1/2)=+1$. By~Dirac~equation~(see~Proposition \ref{dirac-equation-in-electromagnetic-field}):
\begin{equation} \label{dirac-initial-1}
a_1(2x - 1, 2t - 1, u) = \frac{1}{\sqrt{2}}\Big[a_1(2x, 2t - 2, u) + a_2(2x, 2t - 2, u)\Big].
\end{equation}
Next, $(2x, 2t - 2) \underset{4}{\equiv} (0, 2)$ or $(2x, 2t - 2) \underset{4}{\equiv} (2, 0)$. By equality 2 in Theorem \ref{diamond-theorem} we have
\begin{equation} \label{equality-to-zero-1} 
a_2(2x, 2t - 2, u) = 0.
\end{equation}
Next, by Proposition \ref{dirac-equation-in-electromagnetic-field}:
\begin{equation} \label{dirac-deep-1}
a_1(2x, 2t - 2, u) = \frac{1}{\sqrt{2}}\Big[a_1(2x + 1, 2t - 3, u) + a_2(2x + 1, 2t - 3, u)\Big].
\end{equation}
Finally, from (\ref{dirac-initial-1}), (\ref{equality-to-zero-1}), and (\ref{dirac-deep-1}) we get (\ref{new-lattice-recurrent-1}).

Proof of (\ref{new-lattice-recurrent-2}). Let us rewrite the equation in terms of $a_1$ and $a_2$:
$$a_2(2x - 1, 2t - 1, u) = \frac{1}{2}\Big[3a_1(2x - 3, 2t - 3, u) - a_2(2x - 3, 2t - 3, u)\Big].$$
Again, since $(2x - 1, 2t - 1) \underset{4}{\equiv} (1, 1)$ or $(2x - 1, 2t - 1) \underset{4}{\equiv} (3, 3)$, by~Dirac~equation~(see~Proposition~\ref{dirac-equation-in-electromagnetic-field}) we get
\begin{equation} \label{dirac-initial-2}
a_2(2x - 1, 2t - 1, u) = -\frac{1}{\sqrt{2}}\Big[a_2(2x - 2, 2t - 2, u) - a_1(2x - 2, 2t - 2, u)\Big].
\end{equation}
Next, by Proposition~\ref{dirac-equation-in-electromagnetic-field} we get
\begin{equation} \label{dirac-deep-2}
a_2(2x - 2, 2t - 2, u) = \frac{1}{\sqrt{2}}\Big[a_2(2x - 3, 2t - 3, u) - a_1(2x - 3, 2t - 3, u)\Big].
\end{equation}
Since $(2x - 2, 2t - 2) \underset{4}{\equiv} (0, 0)$ or $(2x - 2, 2t - 2) \underset{4}{\equiv} (2, 2)$, by equality 1 in Theorem \ref{diamond-theorem} we get 
\begin{equation} \label{use-theorem-2}
a_1(2x - 2, 2t - 2, u) = \sqrt{2} \cdot a_1(2x - 3, 2t - 3, u).
\end{equation}
Finally, from (\ref{dirac-initial-2}), (\ref{dirac-deep-2}), and (\ref{use-theorem-2}) we get (\ref{new-lattice-recurrent-2}).
\end{proof}

\begin{conjecture}
$$\lim\limits_{t \to \infty} \sum\limits_{x \in \mathbb{Z}} b_1^2(x, t) = \frac{\sqrt{3}}{6}.$$
\end{conjecture}

\section*{Acknowledgements}

Работа была поддержана грантом Фонда развития теоретической физики и математики "Базис" № 21-7-2-19-1.

\end{document}